\documentclass[twoside,12pt]{article}
\usepackage{times,amsmath,amssymb,amsthm,cite}
\newcommand{\subj}[1]{\par\noindent{\bf AMS Subject Classifications: }#1.}
\newcommand{\keyw}[1]{\par\noindent{\bf Keywords: }#1.}
\numberwithin{equation}{section}
\numberwithin{figure}{section}
\newtheorem{theorem}{Theorem}[section]

\theoremstyle{definition}
\newtheorem{definition}[theorem]{Definition}
\newtheorem{example}[theorem]{Example}
\theoremstyle{remark}
\newtheorem{remark}[theorem]{Remark}

\everymath{\displaystyle}
\date{}
\newcommand{\ijde}
{\vspace{-1in}\normalsize\flushleft
International Journal of Difference Equations\\
ISSN 0973-6069, Volume 9, Number 1, pp.~{\thepage--\pageref{lastpage}} (2014)\\
{\tt http://campus.mst.edu/ijde}\\\vspace{1mm}\hrule\vspace{5mm}
\renewcommand\thefootnote{{}}
\footnotetext{\noindent\tt Received November 29, 2013; Accepted February 21, 2014\par
\hspace*{8pt}Communicated by Delfim F. M. Torres}}
\newcommand{\T}{\mathbb{T}}
\newcommand{\R}{\mathbb{R}}
\newcommand{\Z}{\mathbb{Z}}

\newcommand{\rar}{\mbox{$\rightarrow$}}

\begin{document}
\title{\ijde\center\Large\bf Noether's Theorem for Control Problems\\ on Time Scales}
\author{{\bf Agnieszka B. Malinowska} \\
        Faculty of Computer Science\\
Bialystok University of Technology\\
15-351 Bia\l ystok, Poland\\
        {\tt a.malinowska@pb.edu.pl}
\and
        {\bf Moulay Rchid Sidi Ammi}\\
        Department of Mathematics,\\
        AMNEA GROUP\\
Faculty of Sciences and Techniques,\\
Moulay Ismail University, \\
B.P 509, Errachidia, Morocco\\
        {\tt sidiammi@ua.pt}
}
\maketitle
\thispagestyle{empty}

\begin{abstract}
We prove a generalization of Noether's theorem for optimal control problems defined on time scales.
Particularly, our results can be used for discrete-time, quantum, and continuous-time optimal control problems.
The generalization involves a one-parameter family of maps which depend also on the control
and a Lagrangian which is invariant up to an addition of an exact delta differential.
We apply our results to some concrete optimal control problems on an arbitrary time scale.
\end{abstract}

\subj{49K05, 39A12}

\keyw{Time scales, optimal control, conservation laws}


\section{Introduction}

There are close relationships between symmetries and conserved quantities. But how to seek conserved quantities in mechanical systems? Emmy Noether first proposed the famous theorem \cite{noether,ntrans} which is usually formulated in the context of the calculus of variations: it guarantees that the invariance of a variational
integral $\int_a^bL(t,x(t),\dot{x}(t))dt$ with respect to a continuous symmetry transformations that depend on $k$ parameters implies the existence of $k$ conserved quantities along the Euler--Lagrange extremals. Noether's theorem explains conservation laws of mechanical systems: conservation of energy comes from invariance of the system under time translations;
conservation of linear momentum comes from invariance of the system under spacial translations; conservation of angular
momentum reflects invariance with respect to spatial rotations. The result is, however, much more than a theorem.
It is an universal principle, which can be formulated as a precise statement in very different contexts,
and for each such context, under different assumptions \cite{torres2006}. Typically, Noether transformations
are considered to be point-transformations (they are considered to be functions
of coordinates and time), but one can consider more general transformations
depending also on velocities \cite{gelfand,torres}.
In most formulations of Noether's principle, the Noether transformations
keep the integral functional invariant. It is possible,
however, to consider transformations of the problem that keep the invariance of the Lagrangian up to an exact differential,
called a gauge-term \cite{rund,torres2004}. Formulations of Noether's principle are also possible
for optimal control problems \cite{torres,torres1,torres2,torres2006,schaft}.

Time scale calculus is a recent mathematical theory
that unifies two existing approaches to dynamic modelling --- difference and differential equations ---
into a general framework called dynamic models on time scales.
The origins of the idea of time scales calculus date back to the late 1980's when S. Hilger introduced
this notion in his Ph.D. thesis and showed how to
unify continuous time and discrete time dynamical systems \cite{Hilger,Hilger97}. With time this unification
aspect has been supplemented by the extension and generalization features, see e.g., \cite{Agarwal,ahl,Atici06,Atici09,CD:Bohner:2004,livro,loic,MM,zeidan,zeidan2,SA}.
Noether's first theorem has been already extend to the variational calculus on time scales \cite{B:T:08,B:M:T:11}. Therefore, it is natural to ask about a generalization of Noether's theorem for the optimal control problems defined on time scales.
Here we give the answer to this question.

\section{Preliminaries}
\label{sec:prm}
For the convenience of the reader we recall some basic results and notation needed in the sequel.
For the theory of time scales we refer the reader to
\cite{Agarwal,livro,Hilger97}.

A time scale $\T$ is an arbitrary nonempty closed subset of the real
numbers $\R$. The functions $\sigma:\T \rar \T$ and $\rho:\T \rar
\T$ are, respectively, the forward and backward jump operators. The
graininess function on $\T$ is defined by $\mu(t):=\sigma(t)-t$. A point $t\in\mathbb{T}$ is called
right-dense, right-scattered, left-dense, or left-scattered, if
$\sigma(t)=t$, $\sigma(t)>t$, $\rho(t)=t$, or $\rho(t)<t$,
respectively.

Let $\mathbb{T}=[a,b]\cap\mathbb{T}_{0}$ with $a<b$ and
$\mathbb{T}_0$ is a time scale. We define
$\mathbb{T}^\kappa:=\mathbb{T}\backslash(\rho(b),b]$, and
$\mathbb{T}^{\kappa^0} := \mathbb{T}$,
$\mathbb{T}^{\kappa^n}:=\left(\mathbb{T}^{\kappa^{n-1}}\right)^\kappa$
for $n\in\mathbb{N}$.

For a function $f:\T^\kappa\rightarrow \R$ the time scale delta
derivative is denoted by $f^\Delta$ \cite{livro}.
Whenever $f^\Delta$ exists, the following formula holds:
$f^\sigma(t)=f(t)+\mu(t)f^\Delta(t)$,
where we abbreviate $f\circ\sigma$ by $f^\sigma$.
Let $f^{\Delta^{0}} = f$. We define the
$r$th-delta derivative of $f:\mathbb{T}^{\kappa^r}\rightarrow
\R$, $r\in\mathbb{N}$, to be the function $\left(f^{\Delta^{r-1}}\right)^\Delta$,
provided $f^{\Delta^{r-1}}$ is delta differentiable on $\mathbb{T}^{\kappa^r}$.

A function $f:\mathbb{T} \to \mathbb{R}$ is called rd-continuous if
it is continuous at the right-dense points in $\mathbb{T}$ and its
left-sided limits exist at all left-dense points in $\mathbb{T}$. The set of all rd-continuous functions
is denoted by $C_{\mathrm{rd}}$. Similarly, $C^r_{\mathrm{rd}}$ will denote the set of
functions with delta derivatives up to order $r$ belonging to
$C_{\mathrm{rd}}$. A function $f$ is of class $f\in C_{\mathrm{prd}}^r$ if
$f^{\Delta^i}$ is continuous for $i = 0,\ldots,r-1$, and
$f^{\Delta^r}$ exists and is rd-continuous for all, except possibly
at finitely many, $t \in \mathbb{T}^{\kappa^r}$.

A piecewise rd-continuous function $f:\mathbb{T} \to \mathbb{R}$
possess an antiderivative $F$, that is $F^{\Delta}=f$, and in this case the delta
integral is defined by $\int_{c}^{d}f(t)\Delta t=F(d)-F(c)$ for all
$c,d\in\T$. It satisfies
$$\int_t^{\sigma(t)}f(\tau)\Delta\tau=\mu(t)f(t).$$
\begin{example}
\begin{itemize}
\item[a)] If $\T=\R$,
then $\sigma(t)=t=\rho(t)$ and $\mu(t) \equiv 0$ for any $t \in
\R$. The delta derivative, $f^\Delta$, reduces to the standard derivative $f'(t)$ and $\int\limits_{a}^{b}  f(t) \Delta
t=\int\limits_{a}^{b}f(t)dt$, where the integral on the right hand
side is the usual Riemann integral.

\item[b)] If $\T=h\Z$, then $\sigma(t)=t+h$, $\rho(t)=t-h$, and $\mu(t)
\equiv h$ for every $t \in \Z$. The delta derivative, $f^\Delta$, reduces to the $h$-forward difference $\Delta_{h} f(t) = \frac{f(t+h)-f(t)}{h}$ and $\int_{a}^{b}f(t)\Delta t =
\sum_{k=\frac{a}{h}}^{\frac{b}{h}-1} hf(kh)$.

\item[c)] If $\mathbb{T} = q^{\mathbb{N}_0}$, where $q>1$
is a fixed real number, then $\sigma(t) = q t$, $\rho(t) =
q^{-1} t$, and $\mu(t)= (q-1) t$. For the delta derivative and delta integral we get $f^{\Delta}
(t)=\frac{f(q t)-f(t)}{(q-1) t}$ and $\int_{a}^{b}f(t)\Delta t = (q-1) \sum_{t \in [a,b)} t f(t)$, respectively.
\end{itemize}
\end{example}

Let $a,b\in\T$, with $a<b$, and $[a,b]_{\T}:=[a,b]\cap \T$. Consider the nonlinear time scale optimal control problem
\begin{equation}\label{problem:1}
J(x,u) = \int_{a}^{b}
 L(t,x(t),u(t))\Delta t\longrightarrow\min,
\end{equation}
subject to $x\in C^{1}_{\mathrm{prd}}([a,b]_{\T})$, $u\in C_{\mathrm{prd}}([a,\rho (b)]_{\T})$ satisfying
\begin{equation}\label{problem:2}
x^{\Delta}(t)=\varphi(t,x(t),u(t)),\quad t\in [a,\rho (b)]_{\T}.
\end{equation}

\begin{definition}
A pair $(x,u)$ is said to be admissible if it satisfies the dynamic equation \eqref{problem:2}.
\end{definition}

In the next theorem it is assumed that the following regularity conditions hold:
\begin{description}
\item[$(H)$]  The functions $L(t,\cdot,\cdot)$, $\varphi(t,\cdot,\cdot)$ are differentiable in $(x,u)$; and the functions $\frac{\partial L}{\partial x}(t,\cdot,\cdot)$, $\frac{\partial L}{\partial u}(t,\cdot,\cdot)$, $\frac{\partial \varphi}{\partial x}(t,\cdot,\cdot)$, $\frac{\partial \varphi}{\partial u}(t,\cdot,\cdot)$ are continuous at $(x,u)$ uniformly in $t$ and rd-continuous in $t$ for any admissible $(x,u)$; $I+\mu(t)\frac{\partial\varphi}{\partial x}\neq 0 $ for all $t \in[a,\rho (b)]_{\T}$.
\end{description}

\begin{theorem}[See \protect{\cite[Theorem~9.4]{zeidan2}}]
\label{weak:MP}
Assume that $(\bar{x},\bar{u})$ is a weak local minimum for problem \eqref{problem:1}--\eqref{problem:2} such that the assumption $(H)$ holds. Then there exist a constant $\lambda\geq 0$, and a function $\bar{p}\in C^{1}_{\mathrm{prd}}[a,b]_{\T}$, such that $\lambda+\|\bar{p}\|_{C}\neq 0$ (where $\|p\|_{C}:= \max_{t\in [a,\rho (b)]_{\T}}|p(t)|$) and satisfying the following conditions:
\begin{itemize}
\item[(i)] the adjoint equation: for all $t \in [a,\rho (b)]_{\T}$
$$-\bar{p}^{\Delta}(t)=\frac{\partial \bar{\varphi}}{\partial x}\bar{p}^{\sigma}(t)+\lambda \frac{\partial \bar{L}}{\partial x},$$
\item[(ii)] the stationarity condition: for all $t \in [a,\rho (b)]_{\T}$
$$\frac{\partial \bar{\varphi}}{\partial u}\bar{p}^{\sigma}(t)+\lambda \frac{\partial \bar{L}}{\partial u}=0,$$
\end{itemize}
where $\frac{\partial \bar{\varphi}}{\partial x}$, $ \frac{\partial \bar{L}}{\partial x}$, $\frac{\partial \bar{\varphi}}{\partial u}$, $\frac{\partial \bar{L}}{\partial u}$ are evaluated at $(t,\bar{x}(t),\bar{u}(t))$.
\end{theorem}

The Hamiltonian corresponding to problem \eqref{problem:1}--\eqref{problem:2} is defined as follows:
$$H(t,x,u,\lambda,p)=\lambda L(t,x,u)+(p^\sigma)^T\varphi(t,x,u).$$

\begin{remark}
There are two distinct possibilities for the constant $\lambda$ in Theorem~\ref{weak:MP}:
\begin{itemize}
\item[a)] if $\lambda\neq 0$, we say that control $\bar{u}$ is normal, in this situation we may assume that $\lambda=1$;
\item[b)] if $\lambda= 0$, we say that control $\bar{u}$ is abnormal, and in this case the Hamiltonian does not depend on $L$.
\end{itemize}
\end{remark}

For more about normal and abnormal controls, in the context of optimal control problems on time scales, we refer the reader to \cite{loic,zeidan2}

\begin{definition}
We say that $(x,u,\lambda,p)$ with admissible $(x,u)$,
$p\in C_{\mathrm{prd}}^1[a,\rho (b)]_{\T}$, and $\lambda \geq 0$ is an extremal of problem
\eqref{problem:1}--\eqref{problem:2} if the two conditions (i) and (ii) of Theorem~\ref{weak:MP} are satisfied.
\end{definition}


\section{Noether's Theorem without Transformation of Time Variable}

Consider the family of transformations
\begin{equation}\label{tra:fam}
h_s:[a,b]_{\T}\times \R^2\rightarrow \R,
\end{equation}
depending on a parameter $s$, $s\in (-\varepsilon,\varepsilon)$, where  $h_s$ have continuous partial delta derivatives of the first and second order with respect to $s$ and $t$ for all admissible $(x,u)$ (hence, we have equality of mixed partial delta derivatives, see \cite{Bohner+Guseinov2004}), and the value $s=0$ corresponds to the identity transformation: $h_0(t,x,u)=x$ for all $(t,x,u)\in [a,b]_{\T}\times \R^2$. By Taylor's formula we have
\begin{equation}\label{tra:fam:T}
h_s(t,x,u)=h_0(t,x,u)+s\xi(t,x,u)+o(s)=x+s\xi(t,x,u)+o(s),
\end{equation}
where $\xi(t,x,u)=\frac{\partial}{\partial s}h_s(t,x,u)|_{s=0}$.

We define the invariance of \eqref{problem:1}--\eqref{problem:2}  under transformations $h_s(t,x,u)$ up to an exact delta differential as follows.

\begin{definition}\label{def:invariance:1}
 If there exists a function $\Phi_s(t,x,u)$ which has continuous partial delta derivatives of the first and second order with respect to $s$ and $t$ for all admissible $(x,u)$; and for all $s\in (-\varepsilon,\varepsilon)$ and admissible $(x,u)$ there exists a control $u_s(\cdot)\in C_{\mathrm{prd}}[a,\rho (b)]_{\T}$ such that:
\begin{equation}\label{def:invariance:1:1}
L(t,x(t),u(t))+\Phi_s^\Delta(t,x(t),u(t))= L(t,h_s(t,x(t),u(t)),u_s(t)),
\end{equation}
\begin{equation}\label{def:invariance:1:2}
\frac{\Delta}{\Delta t}h_s(t,x(t),u(t))=\varphi(t,h_s(t,x(t),u(t)),u_s(t));
\end{equation}
then problem \eqref{problem:1}--\eqref{problem:2} is invariant under transformation $h_s(t,x,u)$ up to $\Phi_s(t,x,u)$.
\end{definition}

\begin{remark}
We assume that $u_0(\cdot)=u(\cdot)$ and $\frac{\partial}{\partial s}u_s(\cdot)\in C_{\mathrm{prd}}[a,\rho (b)]_{\T}$.
\end{remark}

\begin{remark}
Note that condition \eqref{def:invariance:1:1} is satisfied if and only if
\begin{equation}\label{def:invariance:1:1:1}
\int_{a}^{\beta}\left(L(t,x(t),u(t))+\Phi_s^\Delta(t,x(t),u(t))\right)\Delta t= \int_{a}^{\beta}L(t,h_s(t,x(t),u(t)),u_s(t))\Delta t
\end{equation}
for all $\beta \in [a,b]_{\T}$.
\end{remark}

\begin{theorem}\label{Noether:without:time}
If problem \eqref{problem:1}--\eqref{problem:2} is invariant under transformation $h_s(t,x,u)$ up to $\Phi_s(t,x,u)$, then
\begin{equation}\label{eq:Noether:without:time}
p(t)\frac{\partial}{\partial s}h_s(t,x(t),u(t))|_{s=0}+\lambda\frac{\partial}{\partial s} \Phi_s(t,x(t),u(t))|_{s=0}
\end{equation}
is constant along any extremal of problem \eqref{problem:1}--\eqref{problem:2}.
\end{theorem}

\begin{proof}
Differentiating both sides of equation \eqref{def:invariance:1:1} with respect to $s$, and then setting $s=0$ we obtain
\begin{multline*}
\frac{\partial}{\partial s} \Phi_s^\Delta(t,x(t),u(t))|_{s=0}=\frac{\partial L}{\partial x}(t,x(t),u(t))\frac{\partial }{\partial s}h_s(t,x(t),u(t))|_{s=0}\\+\frac{\partial L}{\partial u}(t,x(t),u(t))\frac{\partial }{\partial s}u_s(t)|_{s=0}.
\end{multline*}
If $(x(\cdot),u(\cdot),\lambda,p(\cdot))$ is an extremal of problem \eqref{problem:1}--\eqref{problem:2}, then it satisfies
$$-p^{\Delta}(t)=\frac{\partial \varphi}{\partial x}p^{\sigma}(t)+\lambda \frac{\partial L}{\partial x}.$$
Therefore,
\begin{equation*}
\lambda\frac{\partial}{\partial s} \Phi_s^\Delta|_{s=0}=-p^\Delta(t)\frac{\partial }{\partial s}h_s|_{s=0}-p^\sigma(t)\frac{\partial \varphi}{\partial x}\frac{\partial }{\partial s}h_s|_{s=0}+\lambda\frac{\partial L}{\partial u}\frac{\partial }{\partial s}u_s(t)|_{s=0}.
\end{equation*}
By the second condition of Theorem~\ref{weak:MP}, we have
\begin{equation*}
\frac{\partial \varphi}{\partial u}p^\sigma(t)\frac{\partial }{\partial s}u_s(t)|_{s=0}+\lambda\frac{\partial L}{\partial u}\frac{\partial }{\partial s}u_s(t)|_{s=0}=0.
\end{equation*}
Combining the last two equations, we arrive to
\begin{equation}\label{prf:1}
\lambda\frac{\partial}{\partial s} \Phi_s^\Delta|_{s=0}+p^\Delta(t)\frac{\partial }{\partial s}h_s|_{s=0}+p^\sigma(t)\frac{\partial \varphi}{\partial x}\frac{\partial }{\partial s}h_s|_{s=0}+\frac{\partial \varphi}{\partial u}p^\sigma(t)\frac{\partial }{\partial s}u_s(t)|_{s=0}=0.
\end{equation}
Observe that, by the assumptions, we have $\frac{\Delta}{\Delta t}\frac{\partial}{\partial s}h_s|_{s=0}=\frac{\partial}{\partial s}\frac{\Delta}{\Delta t}h_s|_{s=0}$. Therefore, from \eqref{def:invariance:1:2} we get
\begin{equation}\label{prf:2}
\frac{\Delta}{\Delta t}\frac{\partial}{\partial s}h_s|_{s=0}=\frac{\partial \varphi}{\partial x}\frac{\partial }{\partial s}h_s|_{s=0}+\frac{\partial \varphi}{\partial u}\frac{\partial }{\partial s}u_s(t)|_{s=0}.
\end{equation}
Inserting \eqref{prf:2} into \eqref{prf:1} we get
\begin{equation}\label{prf:3}
\lambda\frac{\partial}{\partial s} \Phi_s^\Delta|_{s=0}+p^\Delta(t)\frac{\partial }{\partial s}h_s|_{s=0}+ p^\sigma(t)\frac{\Delta}{\Delta t}\frac{\partial}{\partial s}h_s|_{s=0}=0.
\end{equation}
Applying formula $(fg)^\Delta=f^{\sigma}g^{\Delta}+f^{\Delta}g$, we can rewrite \eqref{prf:3} as
\begin{equation*}
\lambda\frac{\partial}{\partial s} \Phi_s^\Delta|_{s=0}+\frac{\Delta}{\Delta t}\left[p(t)\frac{\partial }{\partial s}h_s|_{s=0}\right]=0.
\end{equation*}
Thus,
\begin{equation*}
\frac{\Delta}{\Delta t}\left[\lambda\frac{\partial}{\partial s} \Phi_s|_{s=0}+p(t)\frac{\partial }{\partial s}h_s|_{s=0}\right]=0.
\end{equation*}
Therefore,
\begin{equation*}
\lambda\frac{\partial}{\partial s} \Phi_s|_{s=0}+p(t)\frac{\partial }{\partial s}h_s|_{s=0}
\end{equation*}
is constant along any extremal of problem \eqref{problem:1}--\eqref{problem:2}.
\end{proof}

\begin{example}(Cf. \cite{torres})
Consider the following problem:
\begin{equation}\label{Ex:1:1}
\int_{a}^{b} u^2(t)\Delta t\longrightarrow\min,
\end{equation}
subject to $x\in C^{1}_{\mathrm{prd}}([a,b]_{\T})$, $u\in C_{\mathrm{prd}}([a,\rho (b)]_{\T})$ satisfying
\begin{equation}\label{Ex:1:2}
x^{\Delta}(t)=u(t),\quad t\in [a,\rho (b)]_{\T}.
\end{equation}
This problem is invariant under the transformation $h_s(t,x,u)=x+st$ up to $\Phi_s(t,x,u)=s^2t+2xs$. Indeed, we have
\begin{equation*}
L(u^s)=(u+s)^2=u^2+2us+s^2=u^2+(s^2t+2sx)^\Delta=L(u)+\frac{\Delta}{\Delta t}\Phi_s(t,x,u)
\end{equation*}
and
\begin{equation*}
\frac{\Delta}{\Delta t}h_s(t,x,u)=x^\Delta(t)+s=\varphi(u^s).
\end{equation*}
From Theorem~\ref{Noether:without:time} it follows that $\frac{\Delta}{\Delta t}\left[p(t)t+\lambda 2x(t)\right]=0$ along any extremal of \eqref{Ex:1:1}--\eqref{Ex:1:2}. On the other hand an extremal of \eqref{Ex:1:1}--\eqref{Ex:1:2} should satisfy the following condition:
$p^\Delta (t)=0$ and $p^\sigma(t)+2\lambda u(t)=0$. Therefore, by Theorem~\ref{Noether:without:time}, $\frac{\Delta}{\Delta t}\left[x^\Delta (t)t- x(t)\right]=0$ along any extremal of \eqref{Ex:1:1}--\eqref{Ex:1:2}.
\end{example}

\section{Noether's Theorem with Transformation of Time Variable}

In this subsection we change the time. Thus we consider the optimal control problem on many different time scales. Therefore, we shall assume that problem \eqref{problem:1}--\eqref{problem:2} is defined for all $t\in \R$.
Consider the family of transformations
\begin{equation}\label{tra:fam:2}
h_s(t,x,u)=(h_s^t(t,x,u),h_s^x(t,x,u)):[a,b]_{\T}\times \R^2\rightarrow \R^2,
\end{equation}
depending on a parameter $s$, $s\in (-\varepsilon,\varepsilon)$, where  $h_s^t$, $h_s^x$ have continuous partial delta derivatives of the first and second order with respect to $s$ and $t$ for all admissible $(x,u)$, and the value $s=0$ corresponds to the identity transformation: $(h_0^t(t,x,u),h_0^x(t,x,u))=(t,x)$ for all $(t,x,u)\in [a,b]_{\T}\times \R^2$. By Taylor's formula we have
\begin{eqnarray*}
h_s^t(t,x,u)=h_0^t(t,x,u)+s\zeta(t,x,u)+o(s)=t+s\zeta(t,x,u)+o(s),\\
h_s^x(t,x,u)=h_0^x(t,x,u)+s\xi(t,x,u)+o(s)=x+s\xi(t,x,u)+o(s),
\end{eqnarray*}
where $\zeta(t,x,u)=\frac{\partial}{\partial s}h_s^t(t,x,u)|_{s=0}$ and $\xi(t,x,u)=\frac{\partial}{\partial s}h_s^x(t,x,u)|_{s=0}$.

We assume that for every admissible $(x,u)$ and all $s\in (-\varepsilon,\varepsilon)$ the map $[a,b]\ni t\mapsto \alpha (t):=h_s^t(t,x,u)=t_s\in \R$ is a strictly increasing $C_{\mathrm{rd}}^1$ function and its imagine is again a time scale with the forward shift operator $\bar{\sigma}$, the graininess function $\bar{\mu}$ and the delta derivative $\bar{\Delta}$. Observe that in this case the following holds: $\bar{\sigma}\circ \alpha = \alpha \circ \sigma$ \cite{ahl}.

We define the invariance of \eqref{problem:1}--\eqref{problem:2}  under transformation \eqref{tra:fam:2} up to an exact delta differential as follows.

\begin{definition}\label{def:invariance:2}
Let $I+\bar{\mu}(t_s)\frac{\partial \varphi}{\partial x}\neq 0$ for all
$t_s\in [h_s^t(a,x(a),u(a)), h_s^t(b,x(b),u(b))]^\kappa$. If there exists a function $\Phi_s(t,x,u)$ which has continuous partial delta derivatives of the first and second order with respect to $s$ and $t$ for all admissible $(x,u)$; and for all $s\in (-\varepsilon,\varepsilon)$ and admissible $(x,u)$ there exists a control $u_s(\cdot)\in C_{\mathrm{prd}}[a,\rho (b)]_{\T}$ such that:
\begin{multline}\label{def:invariance:2:1:1}
\int_{h_s^t(a,x(a),u(a))}^{h_s^t(\beta,x(\beta),u(\beta))}L(t_s,h_s^x(t_s,x(t_s),u(t_s)),u_s(t_s))\bar{\Delta}t_s\\
=\int_{a}^{\beta}\left(L(t,x(t),u(t))\Delta t+\Phi_s^\Delta(t,x(t),u(t))\right)\Delta t
\end{multline}
for all $\beta \in [a,b]$;
\begin{equation}\label{def:invariance:2:1:2}
\frac{\bar{\Delta}}{\bar{\Delta} t_s}h_s^x(t_s,x(t_s),u(t_s))=\varphi(t_s,h_s^x(t_s,x(t_s),u(t_s)),u_s(t_s));
\end{equation}
then problem \eqref{problem:1}--\eqref{problem:2} is invariant under transformation \eqref{tra:fam:2} up to $\Phi_s(t,x,u)$.
\end{definition}

\begin{theorem}\label{Noether:with:time}
If problem \eqref{problem:1}--\eqref{problem:2} is invariant under transformation $h_s(t,x,u)$ up to $\Phi_s(t,x,u)$, then
\begin{multline}
\frac{\Delta}{\Delta t}\left[p_k(t)\frac{\partial}{\partial s}h_s^t(t,x(t),u(t))|_{s=0}+p(t)\frac{\partial}{\partial s}h_s^x(t,x(t),u(t))|_{s=0}\right.\\
\left.+\lambda\frac{\partial}{\partial s} \Phi_s(t,x(t),u(t))|_{s=0}\right]=0
\end{multline}
is constant along any extremal of problem \eqref{problem:1}--\eqref{problem:2}.
Moreover, if $\rho \circ \sigma=id_{[a,\rho(b)]_{\T}}$, then
\begin{multline}\label{eq:Noether:with:time}
-H^\rho(t,x(t),u(t),\lambda,p(t))\frac{\partial}{\partial s}h_s^t(t,x(t),u(t))|_{s=0}\\
+p(t)\frac{\partial}{\partial s}h_s^x(t,x(t),u(t))|_{s=0}+\lambda\frac{\partial}{\partial s} \Phi_s(t,x(t),u(t))|_{s=0}
\end{multline}
is constant along any extremal of problem \eqref{problem:1}--\eqref{problem:2}.
\end{theorem}

\begin{proof}
Firstly we show that invariance of \eqref{problem:1}--\eqref{problem:2} in the sense of Definition~\ref{def:invariance:2} is equivalent to invariance of another problem in the sense of Definition~\ref{def:invariance:1}. Let us define
\begin{equation}
\label{problem:3}
\begin{split}
\tilde{J}(k,x,r,u)&=\int_{a}^{b}
\tilde{L}(t;k(t),x(t);r(t),u(t))\Delta t\\
&:=\int_{a}^{b} L(k(t),x(t),u(t))r(t)\Delta t\longrightarrow\min
\end{split}
\end{equation}
subject to
\begin{equation}
\label{problem:4}
\left\{
  \begin{array}{l}
    k^\Delta (t)=r(t)\\
    x^{\Delta}(t)=\tilde{\varphi}(t;k(t),x(t);r(t),u(t)):=\varphi(k(t),x(t),u(t))r(t),\\
  \end{array}
\right.
\end{equation}
$t\in [a,\rho (b)]_{\T}$. The Hamiltonian corresponding to problem \eqref{problem:3}--\eqref{problem:4} has the form
$$\tilde{H}(t,k,x,r,u,\lambda,p_k,p)=\lambda\tilde{L}(t;k,x;r,u)+p_k^{\sigma}r+p^{\sigma}\tilde{\varphi}(t;k,x;r,u).$$
Observe that for $k(t)=t$ we have
$$L(t,x(t),u(t))=\tilde{L}(t;k(t),x(t);r(t),u(t)),$$ so we get $\tilde{J}(k,x,r,u)=J(x,u)$ whenever $k(t)=t$. Moreover, for $k(t)=t$ we have
$$x^{\Delta}(t)=\varphi(t,x(t),u(t))=\tilde{\varphi}(t;k(t),x(t);r(t),u(t)).$$
Now we consider the family of transformations $h_s(t,x,u)=(h_s^t(t,x,u),h_s^x(t,x,u))$. From the invariance of \eqref{problem:1}--\eqref{problem:2}, for $k(t)=t$, we get
\begin{equation}
\begin{split}
\int_{a}^{b}
&\left(\tilde{L}(t;k(t),x(t);r(t),u(t))+\Phi_s^\Delta(t,x(t),u(t))\right)\Delta t\\
&=\int_{a}^{b}
 \left(L(t,x(t),u(t))+\Phi_s^\Delta(t,x(t),u(t))\right)\Delta t\\
 &=\int_{\alpha (a)}^{\alpha (b)}
 L(t_s,h_s^x(t_s,x(t_s),u(t_s)),u_s(t_s))\bar{\Delta}t_s\\
 &=\int_{a}^{b}
L(\alpha (t),h_s^x(\alpha (t),x(\alpha (t)),u(\alpha (t))),u_s(\alpha (t)))\alpha^\Delta (t)\Delta t\\
&=\int_{a}^{b}
\tilde{L}(t;\alpha (t),(h_s^x\circ\alpha) (t,x(t),u(t)),\alpha ^\Delta (t),(u_s\circ \alpha)(t))\Delta t,
\end{split}
\end{equation}
$$\frac{\Delta}{\Delta t}\alpha (k(t))=\alpha ^\Delta (t),$$ and
\begin{equation}
\begin{split}
\frac{\Delta}{\Delta t}(h_s^x\circ \alpha) (t,x(t),u(t))
&=\frac{\bar{\Delta}}{\bar{\Delta} t_s}h_s^x(t_s,x(t_s),u(t_s))\alpha ^\Delta (t)\\
&=\varphi(t_s,h_s^x(t_s,x(t_s),u(t_s)),u_s(t_s))\alpha ^\Delta (t)\\
&=\varphi(\alpha (t),h_s^x(\alpha (t),x(\alpha (t)),u(\alpha (t))),u_s(\alpha (t)))\alpha ^\Delta (t)\\
&=\tilde{\varphi}(t;\alpha (t),(h_s^x\circ \alpha)(t,x(t),u(t));\alpha ^\Delta (t),(u_s\circ\alpha )(t)).
\end{split}
\end{equation}
This means that for $k(t)=t$, if problem \eqref{problem:1}--\eqref{problem:2} is invariant under transformation \eqref{tra:fam:2} up to $\Phi_s(t,x,u)$ in the sense of Definition~\ref{def:invariance:2}, then problem \eqref{problem:3}--\eqref{problem:4} is invariant under transformation \eqref{tra:fam:2} up to $\Phi_s(t,x,u)$ in the sense of Definition~\ref{def:invariance:1}.
Let $(k,x,r,u,\lambda,p_k,p)$ be an extremal of problem \eqref{problem:3}--\eqref{problem:4}. Applying Theorem~\ref{weak:MP} we obtain
$$p^\Delta=-\lambda \frac{\partial \tilde{L} }{\partial x}-p^\sigma\frac{\partial \tilde{\varphi} }{\partial x},$$
$$p_k^\Delta=-\lambda \frac{\partial \tilde{L} }{\partial k} -p^\sigma \frac{\partial \tilde{\varphi} }{\partial k},$$
$$p^\sigma\frac{\partial \tilde{\varphi} }{\partial u}+\lambda\frac{\partial \tilde{L} }{\partial u}=0,$$
$$p_k^\sigma+\lambda\frac{\partial \tilde{L} }{\partial r}+p^\sigma \frac{\partial \tilde{\varphi} }{\partial r}=0.$$
By Theorem~\ref{Noether:without:time}, we get for $k(t)=t$
\begin{multline}
\frac{\Delta}{\Delta t}\left[p_k(t)\frac{\partial}{\partial s}h_s^t(t,x(t),u(t))|_{s=0}+p(t)\frac{\partial}{\partial s}h_s^x(t,x(t),u(t))|_{s=0}\right.\\
\left.+\lambda\frac{\partial}{\partial s} \Phi_s(t,x(t),u(t))|_{s=0}\right]=0.
\end{multline}
Now assume that $\rho \circ \sigma=id_{[a,\rho(b)]_{\T}}$. As $p_k^\sigma=-\lambda L-p^\sigma \varphi$ we have $p_k=-H^\rho$ and condition \eqref{eq:Noether:with:time} holds along any extremal of problem \eqref{problem:1}--\eqref{problem:2} as desired.
\end{proof}

From now on, it is assumed that $\rho \circ \sigma=id_{[a,\rho(b)]_{\T}}$.

\begin{example}[Cf. \cite{gou}]
Consider the following problem:
\begin{equation}
\label{Ex:2:1}
\int_{a}^{b} \left(u_1^2(t)+u_2^2(t)\right)\Delta t\longrightarrow\min,
\end{equation}
subject to $x\in C^{1}_{\mathrm{prd}}([a,b]_{\T})$, $u_1,u_2\in C_{\mathrm{prd}}([a,\rho (b)]_{\T})$ satisfying
\begin{equation}\label{EX:2:2}\left\{
  \begin{array}{l}
    x_1^\Delta (t)=u_1(t)\cos x_3(t),\\
    x_2^{\Delta}(t)=u_1(t)\sin x_3(t),\\
    x_3^\Delta (t)=u_2(t).
  \end{array}
\right.
\end{equation}
The control system \eqref{Ex:2:1}--\eqref{EX:2:2} serves as model for the kinematics of a car \cite{gou,Martin}. This is an autonomous system, so it is invariant under the transformation $h_s^t(t,x,u)=t+s$ up to $\Phi_s\equiv 0$.
From Theorem~\ref{Noether:with:time} it follows that $H\circ\rho$ is constant along any extremal of \eqref{Ex:2:1}--\eqref{EX:2:2}, where
$$H(t,x_1,x_2,x_3,u_1,u_2,\lambda, p_1,p_2,p_3)=\lambda \left(u_1^2+u_2^2\right)+p_1^{\sigma}u_1\cos x_3+p_2^{\sigma}u_1\sin x_3+p_3^{\sigma}u_2.$$
\end{example}

In the next example we apply Theorem~\ref{Noether:with:time} for a problem with abnormal controls.

\begin{example}
Consider the following problem:
\begin{equation}
\label{Ex:3:1}
\int_{0}^{1} u(t)\Delta t\longrightarrow\max,
\end{equation}
subject to $x\in C^{1}_{\mathrm{prd}}([0,1]_{\T})$, $u\in C_{\mathrm{prd}}([0,\rho (1)]_{\T})$ satisfying
\begin{equation}\label{EX:3:2}\left\{
  \begin{array}{l}
    x^\Delta (t)=(u(t)-u^2(t))^2,\\
    x(0)=0,\\
    x(1)=0.
  \end{array}
\right.
\end{equation}
As this system is autonomous, it is invariant under the time transformation $h_s^t(t,x,u)=t+s$ up to $\Phi_s\equiv 0$.
Therefore, by Theorem~\ref{Noether:with:time}, we get that $H\circ\rho$ is constant along any extremal of \eqref{Ex:3:1}--\eqref{EX:3:2}, where
$$H(t,x,u,\lambda, p)=\lambda u+p^{\sigma}(u-u^2)^2.$$ Observe that an extremal of problem \eqref{Ex:3:1}--\eqref{EX:3:2} should satisfy the following conditions:
\begin{equation}\label{Ex:3:3}
p^{\Delta}(t)=0, \quad -4p^{\sigma}(t)u(t)(u(t)-u^2(t))+\lambda=0.
\end{equation}
If we chose $\lambda=0$ and $p=k$, where $k$ is a negative constant, then $(x,u,\lambda,p)=(0,1,0,k)$ or $(x,u,\lambda,p)=(0,0,0,k)$ satisfy conditions \eqref{Ex:3:3}. Note that $x^{\Delta}=(u-u^2)^2\geq0$ and $ x(0)=x(1)=0$ imply that $x^{\Delta}=0$. Hence, if a control is admissible, then we have $u\in\{0,1\}$. In addition, for every admissible $u$, we have
$$\int_{0}^{1} 1\Delta t\geq \int_{0}^{1} u(t)\Delta t.$$
This means that $u=1$ is the optimal control for problem \eqref{Ex:3:1}--\eqref{EX:3:2}.
\end{example}

\begin{remark}
Observe that, for autonomous systems, Theorem~\ref{Noether:with:time} induces that:
\begin{enumerate}
\item $H\circ\rho$ is constant along the optimal path;
\item for $\T=\R$ the well-known result, that the Hamiltonian
$H$ is constant along the optimal path.
\end{enumerate}
\end{remark}

\section*{Acknowledgements}

This work was done during the stay of M. R. Sidi Ammi at the Bialystok University of Technology as a visiting professor.
A. B. Malinowska was supported by the Bialystok
University of Technology grant S/WI/02/2011. The authors are grateful to two anonymous referees
for valuable remarks and comments.
\bibliographystyle{plain}

\label{lastpage}

\end{document}